\documentclass[12pt]{article}
\usepackage[english]{babel} 
\usepackage[utf8]{inputenc} 
\usepackage[T1]{fontenc} 
\usepackage{lmodern} 
\usepackage{amsmath}
\usepackage{amssymb}
\usepackage{amsthm}
\allowdisplaybreaks[1] 
\usepackage[a4paper, margin=1in]{geometry}
\usepackage{hyperref}
\usepackage{xcolor}

\def\RR{\mathbb{R}} 
\def\NN{\mathbb{N}} 
\def\PP{\mathbb{P}} 
\def\EE{\mathbb{E}} 
\def\ind{\mathbf{1}} 
\def\law{\mathcal{L}} 
\def\W{\mathcal{W}} 
\def\N{\mathcal{N}} 
\def\M{\mathcal{M}} 
\def\ii{\mathbf{i}} 

\newtheorem{theorem}{Theorem}
\newtheorem{lemma}[theorem]{Lemma}
\newtheorem{corollary}[theorem]{Corollary}

\theoremstyle{definition}

\title{Uniform propagation of chaos for Kac's 1D particle system}

\author{Roberto Cortez\footnote{CIMFAV, Facultad de Ingenier\'ia, Universidad de Valpara\'iso. General Cruz 222, Valpara\'iso, Chile. E-mail: \texttt{rcortez@dim.uchile.cl}. Supported by Fondecyt Postdoctoral Project 3160250.}}

\begin{document}

\maketitle

\begin{abstract}
In this paper we study Kac's 1D particle system, consisting of the velocities of $N$ particles
colliding at constant rate and randomly exchanging energies.
We prove uniform (in time) propagation of chaos in Wasserstein distance
with explicit polynomial rates in $N$,
for both the squared (i.e., the energy) and non-squared particle system.
These rates are of order $N^{-1/3}$ (almost, in the non-squared case),
assuming that the initial distribution of the limit nonlinear equation
has finite moments of sufficiently high order ($4+\epsilon$ is enough when
using the 2-Wasserstein distance).
The proof relies on a convenient parametrization of the collision
recently introduced by Hauray, as well as on a coupling
technique developed by Cortez and Fontbona.
\end{abstract}

\medskip

{\bf Keywords}: Kinetic theory, Kac particle system, Propagation of chaos.

\medskip

 {\bf  MSC 2010}:  82C40, 60K35.

%

\section{Introduction and main results}
\label{sec:intro}

In this paper we study \emph{Kac's particle system,} introduced
in \cite{kac1956} and later studied for instance in
\cite{carlen-carvalho-loss2003,carlen-carvalho-loss2000,carlen-carvalho-leroux-loss-villani2010,carrapatoso-einav2013}.
It can be described as follows: consider $N$
objects or ``particles'' characterized by their one-dimensional velocities,
subjected to the following binary random ``collisions'': when particles with
velocities $v$ and $v_*$ collide, they acquire new velocities $v'$ and $v_*'$ given by the rule
\begin{equation}
\label{eq:vv*}
(v,v_*) \mapsto (v',v_*') = (v \cos\theta - v_*\sin\theta,  v_* \cos\theta  + v \sin\theta),
\end{equation}
where $\theta \in[0,2\pi)$ is chosen uniformly at random.
This can be seen as a rotation in $\theta$
of the pair $(v,v_*)\in\RR^2$ and, as such, it preserves the energy, i.e., $v^2 + v_*^2 = v'^2 + v_*'^2$.
The system evolves continuously with time $t\geq 0$;
the times between collisions
follow an exponential law with parameter $N/2$ and the two particles
that collide are chosen randomly among all possible pairs, so each particle collides once per unit
of time on average.
The system starts at $t=0$ with some fixed symmetric distribution,
and all the previous random choices are made independently.
This description unambigously determines (the law of) the particle
system, which we denote $\mathbf{V}_t = (V_{1,t},\ldots,V_{N,t})$.

In the pioneering work \cite{kac1956}, Kac proved that for all $t\geq 0$, as $N \to \infty$,
the empirical measure of the system $\frac{1}{N} \sum_i \delta_{V_{i,t}}$
converges weakly to $f_t$ (provided that the convergence holds for $t=0$),
where $(f_t)_{t\geq 0}$ is the collection of probability measures on $\RR$
solving the so-called \emph{Boltzmann-Kac equation}:
\begin{equation}
\label{eq:BK}
\partial_t f_t(v) = \int_0^{2\pi} \int_\RR [f_t(v')f_t(v_*') - f_t(v)f_t(v_*)] dv_* \frac{d\theta}{2\pi}.
\end{equation}
This convergence is now termed \emph{propagation of chaos}, and it has been
extensively studied during the last decades for this and other, more general
kinetic models (especially the Boltzmann equation), see for instance
\cite{sznitman1989,mischler-mouhot2013} and the references therein.

Another interesting feature of this model is its behaviour as $t\to\infty$.
For instance, assuming normalized initial energy, i.e.,
$\sum_i V_{i,0}^2 = N$ a.s.,
it is known that the law of the system converges
exponentially in $L^2$ to its equilibrium, namely, the uniform distribution on
the \emph{Kac sphere} $\{\mathbf{x}\in\RR^N: \sum_i x_i^2 = N\}$,
see \cite{carlen-carvalho-loss2000} and the references therein.
As an alternative approach, one can
couple two copies of the particle system
using the same collision times and the same
angle $\theta$ (i.e., ``parallel coupling''), but with different initial conditions,
to show that the $2$-Wasserstein distance between their laws is non-increasing in time.
However, a simple and better coupling was recently
introduced in \cite{hauray2016}: note first that the post-collisional velocities in \eqref{eq:vv*} can be written as
$(v',v_*') = \sqrt{v^2 + v_*^2} (\cos(\alpha+\theta),\sin(\alpha+\theta))$,
where $\alpha \in (-\pi,\pi]$ is the angle defined by $(v,v_*) = \sqrt{v^2 + v_*^2} (\cos\alpha,\sin\alpha)$,
with the convention that all sums of angles are modulo $2\pi$;
next, note that, since $\theta$ is uniformly chosen in $[0,2\pi)$, so is $\alpha+\theta$, and
then the interaction rule
\begin{equation}
\label{eq:vv*2}
(v,v_*) \mapsto (v',v_*') = \sqrt{v^2 + v_*^2} (\cos(\theta),\sin(\theta))
\end{equation}
generates a system that has the same law than the one described by \eqref{eq:vv*}.
Using this new parametrization of the collision,
one can define a coupling that leads to contraction results in some Wasserstein
metrics, see \cite{hauray2016} for details.

Our goal in this paper is to use the parametrization \eqref{eq:vv*2} in a propagation of chaos context,
in order to obtain explicit (in $N$) and uniform-in-time rates of convergence, as $N\to\infty$,
for the law of the particles towards the solution of \eqref{eq:BK}.
We will quantify this convergence using the \emph{$p$-Wasserstein} distance:
given two probability measures $\mu$ and $\nu$ on $\RR^k$, it is defined as
\[
\W_p(\mu,\nu)
= \left( \inf \EE \frac{1}{k}\sum_{i=1}^k |X_i-Y_i|^p \right)^{1/p},
\]
where the infimum is taken over all random vectors $\mathbf{X} = (X_1,\ldots,X_k)$
and $\mathbf{Y} = (Y_1,\ldots,Y_k)$
such that $\law(\mathbf{X}) = \mu$ and $\law(\mathbf{Y}) = \nu$
(we do not specify the dependence on $k$ in our notation).
We use the \emph{normalized} distance $|\mathbf{x}-\mathbf{y}|_k^p = \frac{1}{k}\sum_i |x_i-y_i|^p$ on $\RR^k$,
which is natural when one cares about the dependence on the dimension.
A pair $(\mathbf{X},\mathbf{Y})$ attaining
the infimum is called an \emph{optimal coupling} and it can be shown
that it always exists. See for instance \cite{villani2009} for background
on optimal coupling and Wasserstein distances.

Let us fix some notation. We denote $E_N = \frac{1}{N} \sum_i V_{i,0}^2$ the (random)
mean initial energy, which is preserved, i.e., $\frac{1}{N} \sum_i V_{i,t}^2 = E_N$ for all $t\geq 0$, a.s.
We also denote $\mathcal{E} = \int_\RR v^2 f_0(dv)$, which itself is preserved
by the flow $(f_t)_{t\geq 0}$.
For a vector $\mathbf{x} = (x_1,\ldots,x_N) \in \RR^N$
we denote by $\mathbf{x}^{(2)} = (x_1^2,\ldots,x_N^2)$
the vector of squares of $\mathbf{x}$,
and we define the (empirical) probability measures
$\bar{\mathbf{x}} = \frac{1}{N}\sum_{j} \delta_{x_j}$ and
$\bar{\mathbf{x}}_i = \frac{1}{N-1}\sum_{j\neq i} \delta_{x_j}$.
Also, for a probability measure $\mu$ on $\RR$, we denote by $\mu^{(2)}$ the measure on $\RR_+$
defined by $\int \phi(v) \mu^{(2)}(dv) = \int \phi(v^2) \mu(dv)$.


\begin{theorem}
	\label{thm:main}
	Assume that $\int_\RR |v|^p f_0(dv) < \infty$ for some $p>4$, $p\neq 8$.
	Let $\gamma = \min(\frac{1}{3},\frac{p-4}{2p-4})$ and $\lambda_N = \frac{1}{4}\frac{N+2}{N-1}$.
	Then, there exists a constant $C$
	depending only on $p$ and $\int_\RR |v|^p f_0(dv)$, such that for all $t\geq 0$,
	\begin{align*}
	\EE \W_2^2(\bar{\mathbf{V}}_t^{(2)}, f_t^{(2)})
	& \leq \frac{C}{N^\gamma} + C \EE (E_N-\mathcal{E})^2 \\
	& \qquad {} +Ce^{-\lambda_N t} \W_2^2(\law(\mathbf{V}_0^{(2)}),(f_0^{(2)})^{\otimes N}).
	\end{align*}
\end{theorem}

This yields a uniform-in-time propagation of chaos in $\W_2^2$ for the energy of the particles.
For instance, assuming that $\int |v|^p f_0(dv) <\infty$ for some $p>8$, the result
gives a rate of order $N^{-1/3}$, provided that
$\EE (E_N-\mathcal{E})^2$ and $\W_2^2(\law(\mathbf{V}_0^{(2)}),(f_0^{(2)})^{\otimes N})$
converge to $0$ at the same rate or faster. Notice also that $\lambda_N$
coincides with the \emph{spectral gap} in $L^2$ of the associated generator
of the particle system, which was computed in \cite{carlen-carvalho-loss2000}
(although with a factor 2 due to a different rate of the collision times).
The restriction $p\neq 8$ comes from the fact that
the proof of Theorem \ref{thm:main} makes use of a general chaocity
result for i.i.d.\ sequences found in \cite[Theorem 1]{fournier-guillin2013};
including the case $p=8$ would produce additional logarithmic terms in the rate,
see \eqref{eq:eps} below.

As in \cite[Corollary 3]{hauray2016}, this $\W_2^2$ propagation of chaos result for the energy
implies the following $\W_4^4$ result for the non-squared system:

\begin{corollary}
	\label{cor:main}
	Let $\mathbf{U}_0 = (U_{1,0},\ldots,U_{N,0})$
	be any vector of i.i.d.\ and $f_0$-distributed random variables,
	and let $\tilde{\gamma} = \frac{p-4}{2p}\ind_{p<8} + \frac{p-4}{3p-8}\ind_{p> 8}$.
	Under the same assumptions as in Theorem \ref{thm:main}, we have for all $t\geq 0$,
	\begin{align*}
	\EE \W_4^4(\bar{\mathbf{V}}_t, f_t)
	&\leq \frac{C}{N^{\tilde{\gamma}}}
	+ C \EE(E_N-\mathcal{E})^2
	+ C e^{-\lambda_N t} \EE \left[\frac{1}{N} \sum_{i=1}^N (V_{i,0}^2-U_{i,0}^2)^2\right] \\
	& \qquad \qquad \qquad \qquad {} + C e^{-t} \EE \left[\frac{1}{N} \sum_{i=1}^N (V_{i,0}-U_{i,0})^4\right].
	\end{align*}
\end{corollary}

Notice that $\tilde{\gamma} < \gamma$ for all $p>4$, thus the rate
obtained is slower than the one of Theorem \ref{thm:main}
(although we can easily deduce a rate $N^{-\gamma}$ in $\W_4^4$ for the law of \emph{one} particle).
For instance, if $f_0$
has finite moment of order $p>8$, Corollary \ref{cor:main} gives a chaos
rate of $N^{-1/4}$ in $\W_4^4$; but if $f_0$ has finite moments of all orders,
it yields a rate of almost $N^{-1/3}$.


Note that when $p$ is close to 4, the chaos rates provided by these
results are very slow. The following theorem provides a good rate
assuming only that $f_0$ has finite moment of order $4+\epsilon$:

\begin{theorem}
	\label{thm:second}
	Assume that $\int_\RR |v|^p f_0(dv) < \infty$ for some $p>4$,
	and that $\sup_N \EE V_{1,0}^4 < \infty$.
	Then, there exists a constant $C$
	depending only on $p$, on $\int_\RR |v|^p f_0(dv)$ and on $\sup_N \EE V_{1,0}^4$,
	such that for all $t\geq 0$,
	\[
	\EE\W_2^2(\bar{\mathbf{V}}_t,f_t)
	\leq \frac{C \log^2 N}{N^{1/3}} + C \W_2^2(\law(\mathbf{V}_0),f_0^{\otimes N}).
	\]
\end{theorem}

To the best of our knowledge, these are the first uniform propagation of chaos results
for Kac's 1D particle system; they will be proven in Section \ref{sec:proofs}.
Similar results for the law of $k$ particles can also be stated.
The rates are explicit and of order $N^{-1/3}$
(almost, in Corollary \ref{cor:main} and Theorem \ref{thm:second}),
assuming enough moments of $f_0$. This is quite reasonable,
given that in general the optimal rate of chaocity for an i.i.d.\ sequence is $N^{-1/2}$,
see \cite[Theorem 1]{fournier-guillin2013}.
Notice that in these results, the initial condition $\mathbf{V}_0$ is not restricted
to have fixed (non-random) mean energy, and can thus be chosen at convenience.
For instance, it can have distribution $f_0^{\otimes N}$, thus the term
$\EE (E_N-\mathcal{E})^2$ in Theorem \ref{thm:main} and Corollary \ref{cor:main}
is easily seen to be of order $1/N$,
while the terms $\W_2^2(\law(\mathbf{V}_0^{(2)}),(f_0^{(2)})^{\otimes N})$,
$\sum_i (V_{i,0}^2-U_{i,0}^2)^2$, $\sum_i (V_{i,0}-U_{i,0})^4$ and
$\W_2^2(\law(\mathbf{V}_0),f_0^{\otimes N})$ all vanish. 
Or one can assume normalized energy (i.e., $E_N = \mathcal{E}$ a.s.),
provided that one can control the remaining terms.

We remark that, although one could use the general
functional techniques of \cite{mischler-mouhot2013} in the present context,
the rates obtained with these techniques are likely to be much slower than the ones presented here.
%
%

The proof of our results mainly relies on the parametrization \eqref{eq:vv*2} introduced
in \cite{hauray2016}, and on a coupling argument developed in \cite{cortez-fontbona2016}
to relate the behaviour of the particle system and the limit jump process
(the \emph{nonlinear process}).
We remark however that,
while the proof of Theorem \ref{thm:main} makes use of the \emph{techniques}
of \cite{hauray2016} and \cite{cortez-fontbona2016},
the proof of Theorem \ref{thm:second} directly combines
the \emph{results} found in these references.

\section{Construction}
\label{sec:construction}

We now give a specific construction of the particle system
and couple it with a suitable system of nonlinear processes,
following \cite{cortez-fontbona2016}.
Consider a Poisson point measure $\N(dt,d\theta,d\xi,d\zeta)$ on $\RR_+\times[0,2\pi)\times [0,N) \times [0,N)$
with intensity $\frac{dt d\theta d\xi d\zeta \ind_\mathcal{G}(\xi,\zeta)}{4 \pi (N-1)}$,
where $\mathcal{G} := \{(\xi,\zeta) \in [0,N)^2: \ii(\xi) \neq \ii(\zeta)\}$
and $\ii(\xi) := \lfloor \xi \rfloor + 1$. In words, the measure $\N$ picks collision times $t\in\RR_+$
at rate $N/2$, and for each such $t$, it also independently samples an angle $\theta$
uniformly at random from $[0,2\pi)$ and a pair $(\xi,\zeta)$ uniformly from the set $\mathcal{G}$
(note that the area of $\mathcal{G}$ is $N(N-1)$). The pair $(\ii(\xi),\ii(\zeta))$ gives
the indices of the particles that jump at each collision.
Using the parametrization \eqref{eq:vv*2},
we define the particle system $\mathbf{V}_t = (V_{1,t},\ldots,V_{N,t})$
as the solution to
\begin{equation}
\label{eq:dVi}
dV_{i,t} = \int_0^{2\pi} \int_{A_i} \left[ \sqrt{V_{i,t^-}^2 + V_{\ii(\xi),t^-}^2 } \cos \theta
- V_{i,t^-} \right] \N_i(dt,d\theta,d\xi),
\end{equation}
for all $i\in\{1,\ldots,N\}$,
where $A_i := [0,N) \setminus [i-1,i)$, and $\N_i$ is the point measure defined as
\begin{equation}
\label{eq:Ni}
\N_i(dt,d\theta,d\xi)
= \N(dt,d\theta,[i-1,i), d\xi) + \N(dt,d\theta - \pi/2,d\xi,[i-1,i)),
\end{equation}
where the $-\pi/2$ is to transform sinus to cosinus.
Clearly, $\N_i$ is a Poisson point measure on $\RR_+ \times [0,2\pi)\times A_i$
with intensity $\frac{dt d\theta d\xi}{2\pi (N-1)}$.
The initial condition $\mathbf{V}_0 = (V_{1,0},\ldots,V_{N,0})$
is some random vector with exchangeable components, independent of $\N$.

The \emph{nonlinear process}
(introduced by Tanaka \cite{tanaka1979} in the context of the Boltzmann equation for
Maxwell molecules) is a stochastic jump-process having marginal laws $(f_t)_{t\geq 0}$,
and it is the probabilistic counterpart of \eqref{eq:BK}.
It represents the trajectory of a fixed particle inmersed in the infinite population,
and it is obtained, for instance, as the solution to \eqref{eq:dVi} when one replaces
$V_{\ii(\xi),t^-}$ (which is a $\xi$-realization of the (random)
measure $\bar{\mathbf{V}}_{i,t^-} = \frac{1}{N-1} \sum_{j\neq i} \delta_{V_{j,t^-}}$)
with a realization of $f_t$.

The key idea, introduced in \cite{cortez-fontbona2016}, is to define,
for each $i\in\{1,\ldots,N\}$, a nonlinear process $U_{i,t}$
that mimics as closely as possible the dynamics
of $V_{i,t}$, which is achieved using a suitable realization of $f_t$
at each collision.
More specifically: the collection $\mathbf{U}_t = (U_{1,t},\ldots,U_{N,t})$ is defined
as the solution to
\begin{equation}
\label{eq:dUi}
dU_{i,t} = \int_0^{2\pi} \int_{A_i} \left[ \sqrt{U_{i,t^-}^2 + F_{i,t}^2(\mathbf{U}_{t^-},\xi) } \cos \theta
- U_{i,t^-} \right] \N_i(dt,d\theta,d\xi),
\end{equation}
for all $i\in\{1,\ldots,N\}$.
Here, $F_{i}$ is a measurable mapping
$\RR_+ \times \RR^N \times A_i \ni (t,\mathbf{x},\xi) \mapsto F_{i,t}(\mathbf{x},\xi)$
such that for all $(t,\mathbf{x})$ and any random variable $\xi$ which is
uniformly distributed on $A_i$, the pair $(x_{\ii(\xi)},F_{i,t}(\mathbf{x},\xi))$
is an optimal coupling between $\bar{\mathbf{x}}_i = \frac{1}{N-1} \sum_{j\neq i} \delta_{x_j}$
and $f_t$ with respect to the cost function $c(x,y) = (x^2-y^2)^2$.
Thus,
\begin{equation}
\label{eq:intFixi2}
\int_{A_i} (x_{\ii(\xi)}^2 - F_{i,t}^2(\mathbf{x},\xi))^2 \frac{d\xi}{N-1}
= \W_2^2(\bar{\mathbf{x}}_i^{(2)},f_t^{(2)}).
\end{equation}
We refer to \cite[Lemma 3]{cortez-fontbona2016} for a proof of existence of such a mapping
(here we use a different cost, but our proof works for any cost that is continuous and bounded from below,
in order to use a measurable selection result of optimal transference plans, such as
\cite[Corollary 5.22]{villani2009}). That lemma also shows that
for any $i\neq j\in\{1,\ldots,N\}$, any random vector $\mathbf{X} \in \RR^N$
with exchangeable components
and any bounded and Borel measurable $\phi:\RR \to \RR$, we have
\begin{equation}
\label{eq:EintphiF}
\EE \int_{i-1}^i \phi(F_{i,t}(\mathbf{X},\xi)) d\xi = \int_\RR \phi(v)f_t(dv).
\end{equation}

The initial conditions $U_{1,0},\ldots,U_{N,0}$ are taken independently and with
law $f_0$. For instance, they can be chosen such that the pair $(\mathbf{V}_0,\mathbf{U}_0)$
is an optimal coupling between $\law(\mathbf{V}_0)$ and $f_0^{\otimes N}$ with respect
to the cost function $(x^2-y^2)^2$, so that
$\EE \frac{1}{N} \sum_i (V_{i,0}^2-U_{i,0}^2)^2 = \W_2^2(\law(\mathbf{V}_0^{(2)}), (f_0^{(2)})^{\otimes N})$
(this is done in the proof of Theorem \ref{thm:main}, but, in 
general, $\mathbf{U}_0$ can be any random vector with law $f_0^{\otimes N}$).

Strong existence and uniqueness of solutions $\mathbf{V}_t = (V_{1,t},\ldots,V_{N,t})$ 
and $\mathbf{U}_t = (U_{1,t},\ldots,U_{N,t})$ for \eqref{eq:dVi}
and \eqref{eq:dUi} are straightforward: since the total rate
of $\N$ is finite over finite time intervals, those equations are nothing but recursions
for the values of the processes at the (timely ordered) jump times.
Also, the collection of pairs $(V_1,U_1),\ldots,(V_N,U_N)$
is clearly exchangeable.

Every $U_{i,t}$ is a nonlinear process, thus $\law(U_{i,t}) = f_t$ for all $t$.
Note however that $U_{i,t}$ and $U_{j,t}$ have simultaneous jumps,
and consequently they are \emph{not} independent.
As in \cite{cortez-fontbona2016},
in order to obtain the desired results,
we will need to show that they become \emph{asymptotically} independent as $N\to\infty$,
which is achieved using a second coupling, see Lemma \ref{lem:decoupling0} below.

\section{Proofs}
\label{sec:proofs}


We will need the following propagation of moments result.

\begin{lemma}
	\label{lem:moments}
	Assume that $\int_\RR |v|^{p} f_0(dv) < \infty$ for some $p\geq 2$.
	Then there exists $C>0$ depending only on $p$ and $\int_\RR |v|^p f_0(dv)$ such that
	$\int_\RR |v|^p f_t(dv) < C$ for all $t\geq 0$.
\end{lemma}
\begin{proof}
	See the proof of \cite[Lemma 5]{cortez-fontbona2016}. 
\end{proof}

\begin{lemma}
	\label{lem:decorrelation}
	Assume that $\int_\RR v^4 f_0(dv) < \infty$.
	Then, there exists a constant $C$ depending only on $\int_\RR v^4 f_0(dv)$, such that for any $i\neq j$,
	\[
	|\textnormal{cov}(U_{i,t}^2,U_{j,t}^2)|
	\leq (1-e^{-t})\frac{C}{N}.
	\]
\end{lemma}

\begin{proof}
	We will estimate $h_t := \EE(U_{i,t}^2 U_{j,t}^2)$.
	From \eqref{eq:dUi} we have
	\begin{equation}
	\label{eq:dUi2Uj2}
	\begin{split}
	dh_t
	&= \EE \int_0^{2\pi} \int_{[0,N)^2}
	\left[
	\ind_{\{\ii(\xi) = i, \ii(\zeta) = j\}} \Delta_1
	+ \ind_{\{\ii(\xi) = j, \ii(\zeta) = i\}} \Delta_2  \right. \\
	& \qquad\qquad {} + \ind_{\{\ii(\xi) = i, \ii(\zeta) \neq j\}} \Delta_3
	+ \ind_{\{\ii(\xi) \neq j, \ii(\zeta) = i\}} \Delta_4 \\
	& \qquad\qquad \left. {} + \ind_{\{\ii(\xi) \neq i, \ii(\zeta) = j\}} \Delta_5
	+ \ind_{\{\ii(\xi) = j, \ii(\zeta) \neq i\}} \Delta_6
	\right]
	\N(dt,d\theta,d\xi,d\zeta),
	\end{split}
	\end{equation}
	where $\Delta_1$ and $\Delta_2$ are the increments of $U_{i,t}^2 U_{j,t}^2$
	when $U_{i,t}$ and $U_{j,t}$ have a simultanous jump, and $\Delta_3,\ldots,\Delta_6$
	are the increments when only one of them jumps. For instance,
	\begin{align*}
	\Delta_1 &= (U_{i,t^-}^2 + F_{i,t}^2(\mathbf{U}_{t^-},\zeta))\cos^2\theta
	(U_{j,t^-}^2 + F_{j,t}^2(\mathbf{U}_{t^-},\xi))\sin^2\theta - U_{i,t^-}^2 U_{j,t^-}^2, \\
	\Delta_3 &= (U_{i,t^-}^2 + F_{i,t}^2(\mathbf{U}_{t^-},\zeta))\cos^2\theta U_{j,t^-}^2 - U_{i,t^-}^2 U_{j,t^-}^2.
	\end{align*}
	We have for the latter:
	\begin{equation}
	\label{eq:EDelta3}
	\begin{split}
	& \EE \int_0^{2\pi} \int_{[0,N)^2} \ind_{\{\ii(\xi) = i, \ii(\zeta) \neq j\}}
	\Delta_3 \N(dt,d\theta,d\xi,d\zeta) \\
	&= \EE \int_0^{2\pi} \int_{A_i} \left[-(1-\cos^2\theta) U_{i,t}^2 U_{j,t}^2 \right. \\
	& \qquad \qquad \qquad \qquad \left. {} + \cos^2\theta F_{i,t}^2(\mathbf{U}_t,\zeta) U_{j,t}^2\right] \frac{dt d\theta d\zeta}{4\pi(N-1)} \\
	&= \left[-\frac{1}{4} h_t + \frac{1}{4}\mathcal{E}^2\right] dt,
	\end{split}
	\end{equation}
	where we have used that $U_{j,t} \sim f_t$ under $\PP$ and $F_{i,t}(\mathbf{U}_t,\zeta) \sim f_t$
	under $\frac{d\zeta\ind_{A_i}(\zeta)}{N-1}$.
	The same identity holds for $\Delta_4, \Delta_5$ and $\Delta_6$.
	On the other hand for $\Delta_1$  we can simply use the Cauchy-Schwarz inequality
	and the fact that $\EE \int_{j-1}^j F_{i,t}^4(\mathbf{U}_t, \zeta) d\zeta = \int v^4 f_t(dv) \leq C$
	(thanks to \eqref{eq:EintphiF} and Lemma \ref{lem:moments}), thus obtaining
	\[
	-\frac{C}{N} dt
	\leq \EE \int_0^{2\pi} \int_{[0,N)^2} \ind_{\{\ii(\xi) = i, \ii(\zeta) = j\}}
	\Delta_1 \N(dt,d\theta,d\xi,d\zeta)
	\leq \frac{C}{N}dt.
	\]
	The same estimate holds true for $\Delta_2$. Using this and \eqref{eq:EDelta3}
	in \eqref{eq:dUi2Uj2}, we deduce that
	$-h_t + \mathcal{E}^2 - \frac{C}{N} \leq \partial_t h_t \leq - h_t + \mathcal{E}^2 + \frac{C}{N}$,
	and multiplying by $e^t$ and integrating yields
	$
	(e^t-1)(\mathcal{E}^2-\frac{C}{N})
	\leq e^t h_t - h_0 \leq
	(e^t-1)(\mathcal{E}^2+\frac{C}{N})
	$.
	But $U_{i,0}$ and $U_{j,0}$ are independent, thus $h_0 = \mathcal{E}^2$,
	and then $\mathcal{E}^2 - (1-e^{-t})\frac{C}{N} \leq h_t \leq \mathcal{E}^2 + (1-e^{-t})\frac{C}{N}$.
	Since $\textnormal{cov}(U_{i,t}^2,U_{j,t}^2) = h_t - \mathcal{E}^2$, the conclusion follows. 
\end{proof}

For a given exchangeable random vector $\mathbf{X}$ on $\RR^N$, denote
$\law^n(\mathbf{X})$ the joint law of its $n$ first components.
The following lemma provides a decoupling property for the system
of nonlinear processes $\mathbf{U}_t$.

\begin{lemma}
	\label{lem:decoupling0}
	Assume $\int_\RR v^4 f_0(dv) < \infty$. Then there exists a constant $C>0$,
	depending only on $\int_\RR v^4 f_0(dv)$,
	such that for all $n\leq N$ and $t\geq 0$,
	\[
	\W_2^2(\law^n(\mathbf{U}_t^{(2)}), (f_t^{(2)})^{\otimes n} )
	\leq C \frac{n}{N}.
	\]
	Also, if $\int_\RR |v|^p f_0(dv) < \infty$ for some $p>4$,
	then there exists a constant $C>0$, depending only on $p$ and $\int_\RR |v|^p f_0(dv)$,
	such that for all $n\leq N$ and $t\geq 0$,
	\[
	\W_4^4(\law^n(\mathbf{U}_t), f_t^{\otimes n} )
	\leq C \left(\frac{n}{N}\right)^{\frac{p-4}{p}}.
	\]
\end{lemma}

\begin{proof}
	The argument uses a coupling construction, as in the proof of \cite[Lemma 6]{cortez-fontbona2016}.
	We repeat the important steps here. First, for all $n\in\{2,\ldots,N\}$, the idea is to construct
	$n$ independent nonlinear processes $\tilde{U}_{1,t},\ldots,\tilde{U}_{n,t}$
	such that $\tilde{U}_{i,t}$ remains close to $U_{i,t}$ on average.
	To achieve this, let $\M$ be an independent copy of the Poisson point
	measure $\N$, and define for all $i\in\{1,\ldots,n\}$
	\begin{equation}
	\label{eq:Mi}
	\begin{split}
	\M_i(dt,d\theta,d\xi)
	&= \N(dt,d\theta,[i-1,i), d\xi) \\
	&\qquad {} + \N(dt,d\theta - \pi/2,d\xi,[i-1,i)) \ind_{[n,N)}(\xi) \\
	&\qquad {} + \M(dt,d\theta - \pi/2,d\xi,[i-1,i)) \ind_{[0,n)}(\xi),
	\end{split}
	\end{equation}
	which is a Poisson point measure on $\RR_+ \times [0,2\pi)\times A_i$
	with intensity $\frac{dt d\theta d\xi}{2\pi (N-1)}$, just as $\N_i$.
	We then define $\tilde{U}_{i,t}$ starting with $\tilde{U}_{i,0} = U_{i,0}$ and solving
	an equation similar to \eqref{eq:dUi}, but using $\M_i$ in place of $\N_i$:
	\begin{equation}
	\label{eq:dtildeUi}
	d\tilde{U}_{i,t} = \int_0^{2\pi} \int_{A_i}
	\left[ \sqrt{\tilde{U}_{i,t^-}^2 + F_{i,t}^2(\mathbf{U}_{t^-},\xi) } \cos \theta
	- \tilde{U}_{i,t^-} \right] \M_i(dt,d\theta,d\xi).
	\end{equation}
	In words, the processes $\tilde{U}_{1,t},\ldots,\tilde{U}_{n,t}$ use the same atoms of $\N$
	that $U_{1,t},\ldots,U_{n,t}$ use,
	except for those that produce a joint jump of $U_{i,t}$ and $U_{j,t}$
	for some $i,j\in\{1,\ldots,n\}$, in which case either $\tilde{U}_{i,t}$ or $\tilde{U}_{j,i}$
	does not jump at that instant. To compensate for the missing jumps,
	additional independent atoms, drawn from $\M$, are added to $\M_i$.
	
	It is clear that $\M_1,\ldots,\M_n$ are independent Poisson point measures.
	Using this and the fact that $F_{i,t}(\mathbf{x},\xi)$ has distribution $f_t$
	when $\xi$ is uniformly distributed on $A_i$, one can
	show that $\tilde{U}_{1,t},\ldots,\tilde{U}_{n,t}$ are independent
	nonlinear processes; see the details in the proof of \cite[Lemma 6]{cortez-fontbona2016}.
	
	Thus, $\W_2^2(\law^n(\mathbf{U}_t^{(2)}), (f_t^{(2)})^{\otimes n} )
	\leq \EE \frac{1}{n} \sum_{i=1}^n (U_{i,t}^2 - \tilde{U}_{i,t}^2)^2$,
	and then, to deduce the first bound,
	it suffices to estimate $h_t := \EE (U_{i,t}^2 - \tilde{U}_{i,t}^2)^2$
	for any fixed $i\in\{1,\ldots,n\}$. From \eqref{eq:dUi} and \eqref{eq:dtildeUi}
	we have
	\begin{equation}
	\label{eq:dUitildeUi}
	\begin{split}
	dh_t
	&= \EE \int_0^{2\pi} \int_{A_i} \Delta_1
	[ \N(dt,d\theta,[i-1,i),d\xi) \\
	&\qquad \qquad \qquad \qquad  {} + \N(dt,d\theta - \pi/2, d\xi, [i-1,i)) \ind_{[n,N)}(\xi)] \\
	& \qquad {} +  \EE \int_0^{2\pi} \int_{A_i} \Delta_2 \N(dt,d\theta - \pi/2, d\xi, [i-1,i) ) \ind_{[0,n)}(\xi) \\
	& \qquad {} +  \EE \int_0^{2\pi} \int_{A_i} \Delta_3 \M(dt,d\theta - \pi/2, d\xi, [i-1,i) ) \ind_{[0,n)}(\xi),
	\end{split}
	\end{equation}
	where $\Delta_1$ is the increment of $(U_{i,t}^2 - \tilde{U}_{i,t}^2)^2$ when
	$U_{i,t}$ and $\tilde{U}_{i,t}$ have a simultaneous jump,
	$\Delta_2$ is the increment when only $U_{i,t}$ jumps,
	and $\Delta_3$ is the increment when only $\tilde{U}_{i,t}$ jumps.
	Thanks to the indicator $\ind_{[0,n)}(\xi)$ and Lemma \ref{lem:moments}, the second and third terms
	in \eqref{eq:dUitildeUi} are easily seen to be of order $C\frac{n}{N}$.
	For the first term, we have
	\begin{align*}
	\Delta_1
	&= \left((U_{i,t^-}^2 + F_{i,t}^2(\mathbf{U}_{t^-},\xi))\cos^2\theta 
	- (\tilde{U}_{i,t^-}^2 + F_{i,t}^2(\mathbf{U}_{t^-},\xi))\cos^2\theta\right)^2 \\
	& \qquad\qquad\qquad\qquad\qquad\qquad\qquad {} - (U_{i,t^-}^2 - \tilde{U}_{i,t^-}^2)^2 \\
	&= -(1-\cos^4\theta) (U_{i,t^-}^2 - \tilde{U}_{i,t^-}^2)^2.
	\end{align*}
	Since $\int_0^{2\pi} (1-\cos^4\theta) \frac{d\theta}{2\pi} = \frac{5}{8}$,
	from \eqref{eq:dUitildeUi} we obtain
	$\partial_t h_t
	\leq -\frac{5}{8} h_t + C\frac{n}{N}$ (we have simply discarded the negative term
	with the indicator $\ind_{[n,N)}(\xi)$ in \eqref{eq:dUitildeUi}),
	and since $h_0 = 0$, the estimate for $\W_2^2$ follows from
	Gronwall's lemma:
	\begin{equation}
	\label{eq:W2lawk}
	h_t
	\leq C(1-e^{-5t/8}) \frac{n}{N}
	\leq C\frac{n}{N}.
	\end{equation}
	
	The estimate for $\W_4^4$ can be reduced to the previous one using an argument
	similar to the proof of \cite[Corollary 3]{hauray2016}: for $i\in\{1,\ldots,n\}$,
	call $S_{i,t}$ the event in which $U_{i,t}$ and $\tilde{U}_{i,t}$ have
	the same sign. On $S_{i,t}$ we have
	\[
	(U_{i,t}-\tilde{U}_{i,t})^4
	\leq (U_{i,t}-\tilde{U}_{i,t})^2 (U_{i,t}+\tilde{U}_{i,t})^2
	= (U_{i,t}^2-\tilde{U}_{i,t}^2)^2,
	\]
	and then, using H\"{o}lder's inequality with $a=\frac{p}{p-4}$ and $b=p/4$, we obtain
	\begin{align*}
	\EE  (U_{i,t}-\tilde{U}_{i,t})^4
	&\leq \EE  \ind_{S_{i,t}} (U_{i,t}^2-\tilde{U}_{i,t}^2)^2
	+ \EE \ind_{S_{i,t}^c} (U_{i,t}-\tilde{U}_{i,t})^4 \\
	&\leq \EE  (U_{i,t}^2-\tilde{U}_{i,t}^2)^2
	+ \PP(S_{i,t}^c)^{1/a}  [\EE (U_{i,t}-\tilde{U}_{i,t})^{4b}]^{1/b}.
	\end{align*}
	The first term in the r.h.s.\ of this inequality is bounded by $Cn/N$ thanks to \eqref{eq:W2lawk}, while
	the expectation in the second term is bounded uniformly on $t$ thanks
	to Lemma \ref{lem:moments}. Also, we have $\PP(S_{i,t}^c) \leq n/(2N)$:
	from \eqref{eq:dUi} and \eqref{eq:dtildeUi} we see that
	when the processes $U_{i,t}$ and $\tilde{U}_{i,t}$ have a joint jump, they
	acquire the same sign (the one of $\cos\theta$), and form \eqref{eq:Ni} and \eqref{eq:Mi},
	it is easy to see that this occurs a proportion $1-n/(2N)$ of the jumps on average.
	With all these, we get
	\[
	\W_4^4(\law^n(\mathbf{U}_t), f_t^{\otimes n})
	\leq \EE \frac{1}{n} \sum_{i=1}^n (U_{i,t}-\tilde{U}_{i,t})^4
	\leq C \left(\frac{n}{N}\right)^{1/a},
	\]
	which proves the estimate for $\W_4^4$.
	
\end{proof}

To prove the following lemma, we will need some preliminaries.
For a probability measure $\mu$ on $\RR$, for any $q \geq 1$ and any $n\in\NN$, define
$\varepsilon_{q,n}(\mu):= \EE \W_q^q(\bar{\mathbf{Z}},\mu)$, where $\mathbf{Z} = (Z_1,\ldots,Z_n)$
is an i.i.d.\ and $\mu$-distributed tuple. The best avaliable estimates
for $\varepsilon_{q,n}(\mu)$ can be found in \cite[Theorem 1]{fournier-guillin2013}:
if $\mu$ has finite $r$-moment for some $r>q$, $r\neq 2q$, then there exists a constant $C$ 
depending only on $q$ and $r$ such that for $\eta = \min(1/2, 1-q/r)$, it holds
\begin{equation}
\label{eq:eps}
\varepsilon_{q,n}(\mu) \leq C \frac{\left(\int |x|^r \mu(dx) \right)^{q/r}}{n^\eta}.
\end{equation}

We will also need the following bound, which is a consequence of \cite[Lemma 7]{cortez-fontbona2016}:
given an exchangeable random vector $\mathbf{X}\in\RR^N$ and a probability
measure $\mu$ on $\RR$,
there exists a constant $C$, depending only on the $q$-moments of $\mu$ and $X_1$,
such that for all $n\leq N$,
\begin{equation}
\label{eq:EWq}
\frac{1}{2^{q-1}} \EE \W_q^q(\bar{\mathbf{X}}, \mu)
\leq \W_q^q(\law^n(\mathbf{X}), \mu^{\otimes n})
+ \varepsilon_{q,n}(\mu)
+ C \frac{n}{N}.
\end{equation}

As a consequence of these estimates and Lemma \ref{lem:decoupling0},
we have:

\begin{lemma}
	\label{lem:decoupling}
	Assume that $\int_\RR |v|^{p} f_0(dv) < \infty$ for some $p>4$, $p\neq 8$.
	Then there exists a constant $C$ depending only on $p$ and
	$\int_\RR |v|^p f_0(dv)$  such that for  $\gamma = \min(\frac{1}{3},\frac{p-4}{2p-4})$ and for all $t\geq 0$,
	\[
	\EE \W_2^2(\bar{\mathbf{U}}_t^{(2)}, f_t^{(2)}) \leq \frac{C}{N^\gamma},
	\]
	and for $\tilde{\gamma} = \frac{p-4}{2p}\ind_{p<8} + \frac{p-4}{3p-8}\ind_{p> 8}$,
	\[
	\EE \W_4^4(\bar{\mathbf{U}}_t, f_t) \leq \frac{C}{N^{\tilde{\gamma}}}.
	\]
	Moreover, the same bounds hold with $\bar{\mathbf{U}}_{i,t}^{(2)}$ in place of $\bar{\mathbf{U}}_t^{(2)}$
	and with $\bar{\mathbf{U}}_{i,t}$ in place of $\bar{\mathbf{U}}_t$, respectively.
\end{lemma}

\begin{proof}

	Using the first part of Lemma \ref{lem:decoupling0}
	and \eqref{eq:eps}-\eqref{eq:EWq} with $\mu = f_t^{(2)}$, $q=2$ and $r = p/2$,
	we obtain $\EE \W_2^2(\bar{\mathbf{U}}_t^{(2)},f_t^{(2)}) \leq C[n^{-\eta} + n/N]$
	for $\eta = \min(1/2, 1-4/p)$
	($C$ depends on the $p/2$ moments of $f_t^{(2)}$, which are controlled uniformly on $t$
	thanks to Lemma \ref{lem:moments}).
	Taking $n = \lfloor N^{1/(1+\eta)} \rfloor$ gives the estimate for $\W_2^2$.
	The estimate for $\W_4^4$ follows similarly: using the second
	part of Lemma \ref{lem:decoupling0} and \eqref{eq:eps}-\eqref{eq:EWq}
	with $\mu = f_t$, $q=4$ and $r = p$, we obtain
	$\EE \W_4^4(\bar{\mathbf{U}}_t,f_t) \leq C[n^{-\eta} + (n/N)^{1/a}]$,
	for $a = \frac{p}{p-4}$ and the same $\eta = \min(1/2, 1-4/p)$.
	Taking $ n= \lfloor N^{1/(1+a \eta)} \rfloor$ gives the desired bound.
	
	The estimates for $\bar{\mathbf{U}}_{i,t}^{(2)}$ and $\bar{\mathbf{U}}_{i,t}$ are obtained similarly. 
\end{proof}

We can now prove Theorem  \ref{thm:main}:

\begin{proof}[Proof of Theorem \ref{thm:main}]
	For some $i\in\{1,\ldots,N\}$ fixed,
	we will estimate the quantity $h_t := \EE(V_{i,t}^2-U_{i,t}^2)^2$.
	Let us first shorten notation: call $V = V_{i,t^-}$, $V_* = V_{\ii(\xi),t^-}$,
	$U = U_{i,t^-}$, $F = F_{i,t}(\mathbf{U}_{t^-},\xi)$, and $U_* = U_{\ii(\xi),t^-}$.
	From \eqref{eq:dVi} and \eqref{eq:dUi}, we have
	\begin{align}
	\notag
	d h_t 
	&= \EE \int_0^{2\pi} \int_{A_i} \left[ ( V^2 + V_*^2 - U^2 - F^2)^2 \cos^4\theta
	- (V^2 - U^2)^2 \right] \N_i(dt,d\theta,d\xi) \\
	\begin{split}
	\label{eq:dVi2Ui2}
	&= \EE \int_0^{2\pi} \int_{A_i}
	\left[ (\cos^4\theta-1)(V^2-U^2)^2 \right. \\
	& \qquad\qquad\qquad
	{} +\cos^4\theta (V_*^2 - U_*^2)^2
	+\cos^4\theta (U_*^2 - F^2)^2 \\
	& \qquad\qquad\qquad
	{} + 2 \cos^4\theta (V^2-U^2+V_*^2-U_*^2)(U_*^2-F^2) \\
	& \qquad\qquad\qquad \left.
	{} + 2 \cos^4\theta (V^2-U^2)(V_*^2-U_*^2) 
	\right] \frac{dt d\theta d\xi}{2\pi(N-1)}.
	\end{split}
	\end{align}
	Clearly $\EE \int_{A_i} (V_*^2 - U_*^2)^2 \frac{d\xi}{N-1} = h_t$, by exchangeability. Thus,
	the first and second terms in the integral of \eqref{eq:dVi2Ui2} yield
	$- h_t dt \int_0^{2\pi} (1-2\cos^4\theta) \frac{d\theta}{2\pi} = -\frac{1}{4} h_t dt$.
	From \eqref{eq:intFixi2}, we have
	$\EE \int_{A_i}(U_*^2-F^2)^2 \frac{d\xi}{N-1} = \EE \W_2^2(\bar{\mathbf{U}}_{i,t}^{(2)},f_t^{(2)})\leq CN^{-\gamma}$,
	thanks to Lemma \ref{lem:decoupling}.
	Using the Cauchy-Schwarz inequality, the third and fourth terms in the integral of \eqref{eq:dVi2Ui2}
	are thus bounded above by $[C N^{-\gamma} + Ch_t^{1/2}N^{-\gamma/2}]dt$. For the remaining term,
	since $\frac{1}{N}\sum_j V_{j,t}^2 = E_N$ for all $t\geq 0$ a.s., we have
	\begin{align*}
	&\EE (V_{i,t}^2 - U_{i,t}^2) \int_{A_i} (V_{\ii(\xi),t}^2 - U_{\ii(\xi),t}^2) d\xi \\
	&= \EE (V_{i,t}^2-U_{i,t}^2)
	\left(-V_{i,t}^2 + U_{i,t}^2 + NE_N - \sum_{j=1}^N U_{j,t}^2 \right) \\
	&\leq -h_t + h_t^{1/2} \left[\EE\left(\sum_{j=1}^N ( U_{j,t}^2 - \mathcal{E})\right)^2 \right]^{1/2}
	+ N h_t^{1/2} \left[\EE(E_N-\mathcal{E})^2\right]^{1/2} \\
	&= -h_t + h_t^{1/2} \left[N\text{var}(U_{i,t}^2) + N(N-1)\textnormal{cov}(U_{i,t}^2,U_{j,t}^2)\right]^{1/2}
	+ N h_t^{1/2} B_N^{1/2},
	\end{align*}
	where in the last line $j\neq i$ is any fixed index,
	and $B_N := \EE(E_N-\mathcal{E})^2$. Thanks to lemmas \ref{lem:moments} and \ref{lem:decorrelation},
	the latter is bounded by $-h_t + Ch_t^{1/2}N^{1/2} + N h_t^{1/2} B_N^{1/2}$;
	thus, the fifth term of \eqref{eq:dVi2Ui2}
	is controlled by $-\frac{3}{4(N-1)}h_tdt + Ch_t^{1/2}[N^{-1/2} + B_N^{1/2}] dt$.
	Gathering all these estimates, we get from \eqref{eq:dVi2Ui2}
	\begin{align*}
	\partial_t h_t
	&\leq -\left(\frac{1}{4} + \frac{3}{4(N-1)} \right) h_t
	+ Ch_t^{1/2}[N^{-\gamma/2} + N^{-1/2}+ B_N^{1/2}]
	+ CN^{-\gamma} \\
	&\leq -\lambda_N h_t + Ch_t^{1/2}[N^{-\gamma/2} + B_N^{1/2}] + CN^{-\gamma}.
	\end{align*}
	Using a version of Gronwall's lemma (see for instance \cite[Lemma 4.1.8]{ambrosio-gigli-savare2008}),
	we obtain
	\begin{equation}
	\label{eq:ht}
	h_t \leq Ce^{-\lambda_N t} h_0 + C N^{-\gamma} + CB_N. 
	\end{equation}
	Finally, note that $\EE \W_2^2(\bar{\mathbf{V}}_t^{(2)},f_t^{(2)})
	\leq 2 \EE \W_2^2(\bar{\mathbf{V}}_t^{(2)},\bar{\mathbf{U}}_t^{(2)})
	+2 \EE \W_2^2(\bar{\mathbf{U}}_t^{(2)},f_t^{(2)})$, 
	and, since $\EE \W_2^2(\bar{\mathbf{V}}_t^{(2)},\bar{\mathbf{U}}_t^{(2)})
	\leq \EE \frac{1}{N} \sum_{j} (V_{j,t}^2 - U_{j,t}^2)^2 = h_t$
	by exchangeability, the conclusion follows from \eqref{eq:ht},
	the first part of Lemma \ref{lem:decoupling},
	and choosing $(\mathbf{V}_0,\mathbf{U}_0)$ as
	an optimal coupling with respect to the cost $(x^2-y^2)^2$, so
	$h_0 = \W_2^2(\law(\mathbf{V}_0^{(2)}), (f_0^{(2)})^{\otimes N})$. 
\end{proof}

\begin{proof}[Proof of Corollary \ref{cor:main}]
	The argument is the same as in the proof of \cite[Corollary 3]{hauray2016},
	and we repeat it here for convenience of the reader.
	From \eqref{eq:dVi} and \eqref{eq:dUi}, it is clear that $V_{i,t}$ and $U_{i,t}$
	have the same sign (the one of $\cos\theta$) after the first jump. And if they
	have the same sign, then
	\[
	(V_{i,t} - U_{i,t})^4
	\leq (V_{i,t} - U_{i,t})^2 (V_{i,t} + U_{i,t})^2
	= (V_{i,t}^2 - U_{i,t}^2)^2.
	\]
	Call $\tau_i$ the time of the first jump of $V_{i,t}$. Then
	\begin{align*}
	\EE (V_{i,t} - U_{i,t})^4
	&\leq \EE \ind_{\{\tau_i \leq t\}} (V_{i,t}^2 - U_{i,t}^2)^2
	+\EE \ind_{\{\tau_i > t\}} (V_{i,t} - U_{i,t})^4 \\
	&\leq \EE (V_{i,t}^2 - U_{i,t}^2)^2
	+\EE \ind_{\{\tau_i > t\}} (V_{i,0} - U_{i,0})^4. 
	\end{align*}
	For the second term we use the fact that $\tau_i$ is independent of $(V_{i,0},U_{i,0})$ and
	has exponential distribution with parameter 1, which gives $e^{-t} \EE (V_{i,0} - U_{i,0})^4$.
	For the first term we simply use \eqref{eq:ht}.
	This yields
	\begin{align*}
	\EE \frac{1}{N} \sum_i (V_{i,t}-U_{i,t})^4
	&\leq CN^{-\gamma} + C e^{-\lambda_N t} \EE \frac{1}{N} \sum_i (V_{i,0}^2-U_{i,0}^2)^2 \\
	& \qquad  {} + C \EE(E_N - \mathcal{E})^2 + C e^{- t} \EE \frac{1}{N} \sum_i (V_{i,0}-U_{i,0})^4.
	\end{align*}
	Finally, we have $\EE\W_4^4(\bar{\mathbf{V}}_t,f_t) \leq C\EE\W_4^4(\bar{\mathbf{V}}_t,\bar{\mathbf{U}}_t)
	+ C\EE\W_4^4(\bar{\mathbf{U}}_t,f_t)$, and the result follows since the first
	term is bounded above by $C\EE \frac{1}{N} \sum_i (V_{i,t}-U_{i,t})^4$
	and using the second part of Lemma \ref{lem:decoupling}
	on the second term (recall that $\tilde{\gamma} < \gamma$). 
\end{proof}

To prove Theorem \ref{thm:second}, we will need the results of \cite{hauray2016}.
They provide exponential contraction rates in $\W_4^4$
for both the particle system and the nonlinear process, which in turn
imply contraction in $\W_2^2$.
More specifically: assuming $\sup_N \EE V_{1,0}^4 < \infty$
and $\int_\RR v^4 f_0(dv)<\infty$, one has for some $\alpha>0$
\begin{equation}
\label{eq:W2contraction}
\W_2^2(\law(\mathbf{V}_t), \mathcal{U}_N) \leq Ce^{-\alpha t}
\quad \text{ and } \quad
\W_2^2(f_t, f_\infty) \leq Ce^{-\alpha t},
\end{equation}
where $\mathcal{U}_N$ and $f_\infty$ are the stationary distributions
for the particle system and nonlinear process, respectively.
Namely, $\mathcal{U}_N$ is the uniform distribution on the sphere
$\{\mathbf{x}\in\RR^N: \frac{1}{N}\sum_i x_i^2 = r^2\}$ with
$r^2$ chosen randomly with the same law as $E_N = \frac{1}{N}\sum_i V_{i,0}^2$,
and $f_\infty$ is the Gaussian distribution with mean $0$ and variance $\mathcal{E} = \int v^2 f_0(dv)$
(note that, although the results of \cite{hauray2016} are stated in the
case $E_N = 1$ a.s., it is easy to generalize them to the case of particle systems
starting a.s.\ with the same random energy).  

Also,
it is easy to verify that
\begin{equation}
\label{eq:W2UNfinfty}
\W_2^2( \mathcal{U}_N, f_\infty^{\otimes N})
\leq CN^{-1/2} + C\W_2^2(\law(\mathbf{V}_0), f_0^{\otimes N}).
\end{equation}
Indeed, given a random vector $\mathbf{Z} = (Z_1,\ldots,Z_N)$ with law $f_\infty^{\otimes N}$
independent of $\mathbf{V}_0$,
call $Q^2 = \frac{1}{N} \sum_{i=1}^N Z_i^2$ and define $Y_i = E_N^{1/2}Z_i / Q$,
so that $\mathbf{Y} = (Y_1,\ldots,Y_N)$ has distribution $\mathcal{U}_N$
thanks to the fact that $f_\infty^{\otimes N}$ is rotation invariant.
A straightforward computation shows that
$\frac{1}{N} \sum_i (Z_i-Y_i)^2 = (Q-E_N^{1/2})^2 \leq 2(Q-\mathcal{E}^{1/2})^2 + 2(E_N^{1/2}-\mathcal{E}^{1/2})^2$,
which is bounded above by $2\W_2^2(\bar{\mathbf{Z}},f_\infty) + 2\W_2^2(\bar{\mathbf{V}}_0,f_0)$,
since $\int v^2 f_\infty(dv) = \int v^2 f_0(dv) = \mathcal{E}$
(in general, for measures $\mu$ and $\nu$ on $\RR$ with $Q_\mu^2 = \int x^2 \mu(dx)$,
one has for any $X \sim \mu$ and $\tilde{X}\sim \nu$:
$\EE(X-\tilde{X})^2 \geq Q_\mu^2 + Q_\nu^2 - 2 Q_\mu Q_\nu = (Q_\mu-Q_\nu)^2$).
This coupling gives
$\W_2^2( \mathcal{U}_N, f_\infty^{\otimes N}) \leq \EE \frac{1}{N} \sum_i (Z_i-Y_i)^2
\leq 2 \EE \W_2^2(\bar{\mathbf{Z}},f_\infty) + 4\EE\W_2^2(\bar{\mathbf{V}}_0,\bar{\mathbf{U}}_0)
+4\EE\W_2^2(\bar{\mathbf{U}}_0,f_0)$,
where the first and third terms are controlled by $CN^{-1/2}$ thanks to \eqref{eq:eps},
and the second term is controlled by
$4\EE \frac{1}{N} \sum_i (V_{i,0} - U_{i,0})^2 = 4\W_2^2(\law(\mathbf{V}_0),f_0^{\otimes N})$,
this time choosing the initial conditions $(\mathbf{V}_0,\mathbf{U}_0)$ as an optimal
coupling with respect to the usual quadratic cost $(x-y)^2$.

We are now ready to prove Theorem \ref{thm:second}:

\begin{proof}[Proof of Theorem \ref{thm:second}]
	The argument combines the contraction results of \cite{hauray2016} and
	the propagation of chaos results of \cite{cortez-fontbona2016}.
	Clearly,
	\begin{equation}
	\label{eq:EW2Vf}
	\begin{split}
	& \EE\W_2^2(\bar{\mathbf{V}}_t , f_t) \\
	&\leq C \EE [ \W_2^2(\bar{\mathbf{V}}_t, \bar{\mathbf{V}}_\infty)
	+ \W_2^2(\bar{\mathbf{V}}_\infty,\bar{\mathbf{Z}}_\infty)
	+ \W_2^2(\bar{\mathbf{Z}}_\infty,f_\infty)
	+ \W_2^2(f_\infty,f_t)].
	\end{split}
	\end{equation}
	Here $\mathbf{V}_\infty$ is a random vector on $\RR^N$
	with law $\mathcal{U}_N$,
	which is also optimally coupled
	to $\mathbf{V}_t$ with respect to the quadratic cost, so $\EE \W_2^2(\bar{\mathbf{V}}_t, \bar{\mathbf{V}}_\infty)
	\leq \EE \frac{1}{N}\sum_i (V_{i,t} - V_{i,\infty})^2 = \W_2^2(\law(\mathbf{V}_t), \law(\mathbf{V}_\infty))$.
	Thus, the first and fourth term are bounded by $Ce^{-\alpha t}$,
	thanks to \eqref{eq:W2contraction}.
	Also, we have chosen $\mathbf{Z}_\infty$ with law $f_\infty^{\otimes N}$ and being optimally coupled
	to $\mathbf{V}_\infty$, so for the second term of \eqref{eq:EW2Vf} we have
	$\EE\W_2^2(\bar{\mathbf{V}}_\infty,\bar{\mathbf{Z}}_\infty)
	\leq \W_2^2(\mathcal{U}_N,f_\infty^{\otimes N})$,
	which is controlled using \eqref{eq:W2UNfinfty}.
	The third term is controlled by $CN^{-1/2}$, thanks to \eqref{eq:eps}.
	With all these estimates, we obtain from \eqref{eq:EW2Vf}:
	\begin{equation}
	\label{eq:EW2Vf_B}
	\EE \W_2^2(\bar{\mathbf{V}}_t,f_t)
	\leq Ce^{-\alpha t} + C\W_2^2(\law(\mathbf{V}_0),f_0^{\otimes N}) + CN^{-1/3}
	\end{equation}
	for some $\alpha>0$. On the other hand, from \cite[Theorem 1]{cortez-fontbona2016}
	we have
	\begin{equation}
	\label{eq:EW2Vf_C}
	\EE \W_2^2(\bar{\mathbf{V}}_t,f_t)
	\leq C\W_2^2(\law(\mathbf{V}_0),f_0^{\otimes N}) + C(1+t)^2N^{-1/3}.
	\end{equation}
	(In \cite{cortez-fontbona2016} the initial distribution of the particle system
	was chosen as $f_0^{\otimes N}$, but the extension to any exchangeable initial condition
	is straightforward). Finally, the result is obtained from \eqref{eq:EW2Vf_B} and \eqref{eq:EW2Vf_C}
	adjusting $t$ and $N$ conveniently: take $t_* = \frac{\log N}{3\alpha}$, so
	\eqref{eq:EW2Vf_B} yields
	$\EE \W_2^2(\bar{\mathbf{V}}_t,f_t) \leq C\W_2^2(\law(\mathbf{V}_0),f_0^{\otimes N}) + CN^{-1/3}$ for $t\geq t_*$,
	whereas \eqref{eq:EW2Vf_C} gives
	$\EE \W_2^2(\bar{\mathbf{V}}_t,f_t) \leq C\W_2^2(\law(\mathbf{V}_0),f_0^{\otimes N}) + C N^{-1/3}\log^2 N$
	for $t\leq t_*$. The result follows. 
\end{proof}

\medskip

\textbf{Acknowledgements.}
The author thanks Joaquin Fontbona and Jean-Fran\c{c}ois Jabir for
very useful suggestions and corrections of earlier versions of this manuscript.

\bibliographystyle{plain}
\bibliography{references.bib}{}

%


\end{document}